\documentclass[12pt]{amsart}
\usepackage[margin=2cm]{geometry}
\usepackage{wasowicz_hyperref}
\usepackage{times}

\theoremstyle{definition}
\newtheorem{ex}{Experiment}

\info{Manuscript}

\author{\SW}
\address{\SWaddr}
\email{\SWmail}

\title{Adaptive integration of $5$-convex and $5$-concave functions}
\date{\today}

\begin{document}

\begin{abstract}
This paper introduces an adaptive numerical integration method specifically designed for $5$-convex and $5$-concave functions of class $C^6[a,b]$. The approach is based on a refined inequality involving the three-point Gauss quadrature ($\mathcal{G}$) and the four-point Lobatto quadrature ($\mathcal{L}$). We utilize a specific linear combination, $Q = \frac{3}{4}\mathcal{G} + \frac{1}{4}\mathcal{L}$, and demonstrate that for functions with a sixth derivative of constant sign, the approximation error is effectively controlled by the difference between these two quadratures. We provide a rigorous justification for the stopping criterion of the adaptive algorithm. Numerical experiments, including the approximation of the integrals of a reciprocal function as well as of the exponential function, show that the proposed method significantly outperforms existing adaptive techniques designed for lower-order convexity, requiring substantially fewer subintervals to achieve high precision (up to $10^{-16}$).
\end{abstract}

\maketitle

\section{Remarks on adaptive methods of approximate integration}
Numerical integration of a function $f\colon [a,b]\to\R$ is typically performed using a \emph{quadrature rule}:
\[
\int_a^b f(x) dx \approx Q[f;a,b] := \sum_{k=1}^m w_k f\bigl(\lambda_k a + (1-\lambda_k)b\bigr),
\]
where the coefficients $\lambda_1, \dots, \lambda_m \in [0,1]$ and the weights $w_1, \dots, w_m \in \R$ (usually positive) are fixed parameters. A well-known example is the \emph{simple Simpson's rule}:
\[
\Smp[f;a,b] = \frac{b-a}{6} \biggl[f(a) + 4f\biggl(\frac{a+b}{2}\biggr) + f(b)\biggr].
\]
Every quadrature rule induces a \emph{composite rule}, created by subdividing the interval $[a,b]$ into $n$ subintervals with equally spaced end-points $a=x_0 < x_1 < \dots < x_n=b$:
\[
\int_a^b f(x) dx \approx \sum_{k=1}^n Q[f;x_{k-1}, x_k].
\]
Let $\varepsilon > 0$ be a fixed tolerance. In most cases, a simple quadrature rule is insufficient, as $|Q[f;a,b] - \I[f;a,b]| > \varepsilon$. To achieve the desired precision, we apply the composite rule $Q_n$ with $n$ large enough so that:
\[
|Q_n[f;a,b] - \I[f;a,b]| < \varepsilon.
\]
Such a method is called \emph{adaptive}. To determine $n$ for functions that are sufficiently regular, error bounds can be utilized. For instance, for the composite Simpson's rule over $f \in C^4[a,b]$, the following estimation is widely used:
\[
|\Smp_n[f;a,b] - \I[f;a,b]| \xle \frac{(b-a)^5}{2880n^4} \sup \bigl\{ |f^{(4)}(x)| : x \in [a,b] \bigr\}.
\]
If the right-hand side is less than $\varepsilon$, a suitable $n$ can be easily determined. However, this approach can be complicated due to difficulties in estimating the maximum of $|f^{(4)}|$, or even impossible if $f$ lacks sufficient regularity.

Another approach to adaptive integration is based on the successive incrementation of $n$. Various \emph{stopping criteria} exist in the literature. Lyness \cite{Lyn69} proposed using local errors:
\[
\varepsilon_k = \frac{1}{15} \biggl| \Smp\Bigl[f;x_k, \frac{x_k+x_{k+1}}{2}\Bigr] + \Smp\Bigl[f; \frac{x_k+x_{k+1}}{2}, x_{k+1}\Bigr] - \Smp[f;x_k, x_{k+1}] \biggr|.
\]
Another criterion, proposed by Clenshaw and Curtis \cite{CleCur60} and investigated by Rowland and Varol \cite{RowVar72}, uses the stopping inequality:
\[
|\Smp_{2n}[f;a,b] - \I[f;a,b]| \xle |\Smp_{2n}[f;a,b] - \Smp_n[f;a,b]|,
\]
valid for all $n \in \mathbb{N}$ and functions $f \in C^4[a,b]$ with $f^{(4)}$ of a constant sign. An interesting survey of such criteria was provided by Gonnet \cite{Gon12} (also available on the arXiv repository, 	\href{https://arxiv.org/abs/1003.4629}{\texttt{arXiv:1003.4629}}).

Following these ideas, an adaptive method based on Simpson and Chebyshev quadratures was introduced in \cite{Was20}:
\[
C[f;a,b] = \frac{b-a}{3} \biggl[ f\biggl(\frac{2+\sqrt{2}}{4}a + \frac{2-\sqrt{2}}{4}b\biggr) + f\biggl(\frac{a+b}{2}\biggr) + f\biggl(\frac{2-\sqrt{2}}{4}a + \frac{2+\sqrt{2}}{4}b\biggr) \biggr]
\]
within the class of 3-convex functions. In this paper, we present a similar method for 5-convex functions. Numerical experiments show that it requires significantly fewer computations than the method described in \cite{Was20}.

\section{Higher-order convexity}

Given that our study focuses on $5$-convex functions belonging to the $C^6$ class, it is appropriate to provide a concise overview of the general theory of higher-order convexity. The idea of examining functions characterized by non-negative divided differences was originally introduced by Hopf~\cite{Hop26}. This line of research was further expanded by Popoviciu~\cite{Pop34, Pop44}, who explored the characteristics of this family of functions. According to his definition, a function $f$ is classified as $n$-\emph{convex} provided that
\[
 [x_0, x_1, \dots, x_{n+1}; f] \xge 0,
\]
where the divided differences are determined through the recursive formula $[x_0; f] = f(x_0)$ and
\[
 [x_0, x_1, \dots, x_{n+1}; f] = \frac{[x_1, \dots, x_{n+1}; f] - [x_0, \dots, x_n; f]}{x_{n+1} - x_0}.
\]
Popoviciu established several fundamental properties of these mappings. Notably, he demonstrated that $n$-convexity on an interval $[a, b]$ implies the existence of the $(n-1)$-th derivative on $(a, b)$. It should be noted, however, that while $x \mapsto x^n_+$ (where $x_+ = \max\{x, 0\}$) is $n$-convex on $\mathbb{R}$, it lacks an $n$-th derivative at the origin. Furthermore, he proved (\cite[p.~18]{Pop34}) that for odd $n$, the $n$-convexity is preserved under composition with affine transformations. 

Another significant finding by Popoviciu is that the non-negativity of the $(n+1)$-th derivative on $[a, b]$ is a sufficient condition for a function to be $n$-convex. Within this framework, $1$-convexity coincides with the standard definition of a convex function. Analogously, $f$ is termed $n$-\emph{concave} if the function $-f$ satisfies the condition for $n$-convexity. Consequently, functions whose sixth derivative maintains a constant sign (as discussed earlier) are either $5$-convex or $5$-concave.

\section{Introduction}

In this paper, we focus on the adaptive approximation of integrals of $5$-convex or $5$-concave functions of class $C^6[a,b]$, i.e., functions whose sixth derivative is continuous and has a constant sign on the interval of integration $[a,b]$. We will employ the \emph{three-point Gauss quadrature}
\begin{align*}
 \G[f;a,b]&=\frac{b-a}{18}\Biggl[
    5 f\biggl(\frac{5+\sqrt{15}}{10}a+\frac{5-\sqrt{15}}{10}b\biggr)+
    8 f\biggl(\frac{a+b}{2}\biggr)+
    5 f\biggl(\frac{5-\sqrt{15}}{10}a+\frac{5+\sqrt{15}}{10}b\biggr)
   \Biggr]
\intertext{and the \emph{four-point Lobatto quadrature}}
 \Lob[f;a,b]&=\frac{b-a}{12}\Biggl[
    f(a)+
    5 f\biggl(\frac{5+\sqrt{5}}{10}a+\frac{5-\sqrt{5}}{10}b\biggr)+
    5f\biggl(\frac{5-\sqrt{5}}{10}a+\frac{5+\sqrt{5}}{10}b\biggr)+
    f(b)
   \Biggr].
\end{align*}

Bessenyei and P\'ales \cite{BesPal02} proved that if a function $f\colon [a,b]\to\R$ is $5$-convex, then
\begin{equation}\label{ineq:BP}
 \G[f;a,b] \xle \int_a^b f(x) dx \xle \Lob[f;a,b].
\end{equation}
Together with Komisarski \cite[Theorem~5]{KomWas17}, we refined this inequality for $5$-convex functions of class $C^6[-1,1]$, proving that
\begin{equation}\label{ineq:KomWas17}
 0 \xle \int_{-1}^1 f(x) dx - \G[f;-1,1] \xle \Lob[f;-1,1] - \int_{-1}^1 f(x) dx.
\end{equation}
Note that this inequality can be immediately extended to any interval $[a,b]$. Let $f\colon [a,b]\to\R$ be a $5$-convex function of class $C^6[a,b]$. The affine transformation
\begin{equation}\label{subst}
 t = \frac{2}{b-a}(x-a)-1,\quad x\in[a,b]. 
\end{equation}
maps the interval $[a,b]$ onto $[-1,1]$, and the function $g\colon[-1,1]\to\R$ defined by
\[
 g(t) = \frac{b-a}{2}f\biggl(\frac{a+b}{2}+\frac{b-a}{2}t\biggr)
\]
is $5$-convex of class $C^6[-1,1]$, as $g^{(6)}$ is continuous and non-negative on $[-1,1]$. Applying \eqref{ineq:KomWas17} to the function $g$, simple calculations lead to the inequality
\[
 0 \xle \int_a^b f(x) dx - \G[f;a,b] \xle \Lob[f;a,b] - \int_a^b f(x) dx.
\]
From this, we immediately obtain:
\begin{prop}\label{prop:GL}
 If the function $f\in C^6[a,b]$ is $5$-convex, then
 \[
  \G[f;a,b] \xle \int_a^b f(x) dx \xle \frac{\G[f;a,b] + \Lob[f;a,b]}{2}.
 \]
\end{prop}
This serves as the starting point for our considerations.

\section{The Adaptive Method}

Let
\[
 Q[f;a,b] = \frac{\G[f;a,b] + \frac{\G[f;a,b] + \Lob[f;a,b]}{2}}{2} = \frac{3}{4}\G[f;a,b] + \frac{1}{4}\Lob[f;a,b].
\]
A direct consequence of this definition and Proposition~\ref{prop:GL} is the following corollary, which is central to the construction of our adaptive method.
\begin{cor}\label{cor:Q}
 If the function $f\in C^6[a,b]$ is $5$-convex, then
 \[
  \biggl|
   \int_a^b f(x) dx - Q[f;a,b]
  \biggr| \xle
  \frac{1}{4}\bigl(\Lob[f;a,b] - \G[f;a,b] \bigr).
 \]
\end{cor}
\begin{proof}
 If $\alpha \xle t \xle \beta$, then $\Bigl|t-\frac{\alpha+\beta}{2}\Bigr| \xle \frac{\beta-\alpha}{2}$ (an easy geometric observation). It suffices to set $\alpha=\G[f;a,b]$ and $\beta=\frac{\G[f;a,b] + \Lob[f;a,b]}{2}$.
\end{proof}

Inequality~\eqref{ineq:BP} was refined in Proposition~\ref{prop:GL} by taking the arithmetic mean of the quadratures $G[f;a,b]$ and $\Lob[f;a,b]$. We will now show that taking a further arithmetic mean in the form of $Q[f;a,b]$ does not inherently improve our results, as there is no fixed relation between the integral of a $5$-convex function and the quadrature $Q[f;a,b]$. To this end, let us consider two functions $f(x)=(x-0.6)^7_+$ and $g(x)=(x-0.7)^7_+$. Since $(x^n_+)' = nx^{n-1}_+$ for $n\in\mathbb{N}$, $n\xge 2$, it follows that $f,g\in C^6[-1,1]$. Their sixth derivatives are clearly non-negative, so $f,g$ are $5$-convex. It can be computed that
\begin{align*}
 \int_{-1}^1 f(x) dx &\approx 8.2\cdot 10^{-5},\quad Q[f;-1,1]\approx 7.0\cdot 10^{-5},\\[1ex]
 \int_{-1}^1 g(x) dx &\approx 8.2\cdot 10^{-6},\quad Q[g;-1,1]\approx 9.1\cdot 10^{-6}.
\end{align*}
Thus,
\[
 \int_{-1}^1 f(x) dx > Q[f;-1,1] \quad\text{and}\quad \int_{-1}^1 g(x) dx < Q[f;-1,1].
\]

If the interval $[a,b]$ is relatively long, estimating the precision of the integral approximation $\int_a^b f(x) dx$ using quadrature $Q[f;a,b]$ based on Corollary~\ref{cor:Q} may be imprecise. In such cases, as is standard in numerical analysis, we resort to \emph{composite quadratures}.

For $n\in\mathbb{N}$, we divide the interval $[a,b]$ into $n$ subintervals of equal length with points $a=x_0 < x_1 < \dots < x_n=b$ and define:
\begin{align*}
 \G_n[f;a,b] &= \sum_{k=1}^n \G[f;x_{k-1},x_k],\\[1ex]
 \Lob_n[f;a,b] &= \sum_{k=1}^n \Lob[f;x_{k-1},x_k],\\[1ex]
 Q_n[f;a,b] &= \sum_{k=1}^n Q[f;x_{k-1},x_k].
\end{align*}

Error forms for the quadratures $\G[f;-1,1]$ and $\Lob[f;-1,1]$ for $f\in C^6[-1,1]$ can be found in almost any numerical analysis textbook (for a quick reference, see e.g., Abramowitz, Stegun~\cite[pp. 887, 888]{AbrSte65}). By applying transformation~\eqref{subst}, these can be rewritten for the interval $[a,b]$, and the errors of the composite quadratures can then be estimated. Thus, for $f\in C^6[a,b]$, we obtain:
\begin{align*}
 \biggl|\int_a^b f(x) dx - \G_n[f;a,b] \biggr| &\xle \frac{(b-a)^7}{2016000\,n^6}\|f^{(6)}\|_{\infty},\\[1ex]
 \biggl|\int_a^b f(x) dx - \Lob_n[f;a,b] \biggr| &\xle \frac{(b-a)^7}{1512000\,n^6}\|f^{(6)}\|_{\infty}
\end{align*}
which implies that
\[
 \lim\limits_{n\to\infty} \G_n[f;a,b] = \lim\limits_{n\to\infty} \Lob_n[f;a,b] = \int_a^b f(x) dx,
\]
which in turn gives us
\begin{equation}\label{lim}
 \lim\limits_{n\to\infty}\bigl| \Lob_n[f;a,b] - \G_n[f;a,b] \bigr| = 0.
\end{equation}

The following result is the heart of our method.

\begin{thm}\label{th:adaptive}
 If the function $f\in C^6[a,b]$ is $5$-convex or $5$-concave, then
 \[
  \biggl|
   \int_a^b f(x) dx - Q_n[f;a,b]
  \biggr| \xle
  \frac{1}{4}\bigl|\Lob_n[f;a,b] - \G_n[f;a,b] \bigr|.
 \]
\end{thm}
\begin{proof}
 First, assume that $f\in C^6[a,b]$ is $5$-convex. Given the construction of the quadratures $\G_n[f;a,b]$ and $\Lob_n[f;a,b]$, by Corollary~\ref{cor:Q} we have:
 \begin{align*}
  \biggl| \int_a^b f(x) dx - Q_n[f;a,b] \biggr| &=
  \biggl| \sum_{k=1}^n\biggl( \int_{x_{k-1}}^{x_k} f(x) dx - Q[f;x_{k-1},x_k]\biggr)\biggr|\\
  &\xle \sum_{k=1}^n \biggl| \int_{x_{k-1}}^{x_k} f(x) dx - Q[f;x_{k-1},x_k]\biggr|\\
  &\xle \frac{1}{4}\sum_{k=1}^n\bigl(\Lob[f;x_{k-1},x_k]-\G[f;x_{k-1},x_k]\bigr)\\
  &=\frac{1}{4}\bigl(\Lob_n[f;a,b]-\G_n[f;a,b]\bigr)
  =\frac{1}{4}\bigl|\Lob_n[f;a,b] - \G_n[f;a,b] \bigr|.
 \end{align*}
 If the function $f$ is $5$-concave, we apply the above inequality to the $5$-convex function $-f$. Taking into account the linearity of the integral and the operators $Q_n, \Lob_n, \G_n$, we arrive at:
 \[
  \biggl| \int_a^b f(x) dx - Q_n[f;a,b] \biggr| \xle \frac{1}{4}\bigl(\G_n[f;a,b]-\Lob_n[f;a,b]\bigr)=\frac{1}{4}\bigl|\Lob_n[f;a,b] - \G_n[f;a,b] \bigr|.
 \]
\end{proof}

From Theorem~\ref{th:adaptive}, it immediately follows:
\begin{cor}\label{cor:adaptive}
 If the function $f\in C^6[a,b]$ is $5$-convex or $5$-concave and
 \[
  \bigl|\Lob_n[f;a,b]-\G_n[f;a,b]\bigr| \xle 4\varepsilon,
 \]
 then
 \[
  \biggl|\int_a^b f(x) dx - Q_n[f;a,b]\biggr| \xle \varepsilon.
 \]
\end{cor}

The above inequality will serve as the \emph{stopping criterion} for our adaptive integration method for $5$-convex or $5$-concave functions of class $C^6[a,b]$. The algorithm proceeds in three steps:

\begin{enumerate}[1.]
 \item
  Take a function $f\in C^6[a,b]$ that is either $5$-convex or $5$-concave.
 \item
  Set the desired approximation accuracy $\varepsilon > 0$ and look for $n\in\mathbb{N}$ such that
  \[
   \biggl|\int_a^b f(x) dx - Q_n[f;a,b]\biggr| \xle \varepsilon.
  \] 
  Calculations begin with $n=1$.
 \item
  Increase $n$ until the inequality $\bigl|\Lob_n[f;a,b]-\G_n[f;a,b]\bigr| \xle 4\varepsilon$ is satisfied. Note that such an $n$ exists due to~\eqref{lim}, so the algorithm is guaranteed to terminate. By Corollary~\ref{cor:adaptive}, we obtain an approximation of the integral of $f$ with the specified accuracy.
\end{enumerate}

\section{Numerical Experiments}

In the last part of the paper, we will see how our adaptive method performs in practice. For two $5$-convex functions of class $C^6$, we will examine how many subintervals are required to achieve the desired integral approximation accuracy. We use the same functions as in \cite{Was20}, where an analogous adaptive method based on composite Simpson and Chebyshev quadratures was implemented for $3$-convex or $3$-concave functions. The functions considered here are, of course, both $5$-convex and $3$-convex. Calculations were performed using SageMath version 9.5~\cite{sagemath95}.

\begin{ex}
Let $f(x)=\frac{1}{x}$, $x\in[1,2]$. We examine how many subdivisions are required to approximate the integral $\int_1^2 \frac{dx}{x} = \ln 2$ using the algorithm based on Corollary~\ref{cor:adaptive} for accuracies $\varepsilon=10^{-k}$ ($k=1,2,\dots,16$). The results are compared with the algorithm described in \cite{Was20}.

\begin{center}
\begin{tabular}{lp{4cm}p{4cm}}
 \hline
 $f(x)=\frac{1}{x}$ & \multicolumn{2}{c}{Number of subdivisions of $[1,2]$}\\\cline{2-3}
 Accuracy & Algorithm based\newline on Corollary~\ref{cor:adaptive}  & Algorithm for\newline $3$-convex functions,\newline see~\cite[Experiment 1]{Was20}\\
 \hline
 $10^{-1}$ & $1$  & $1$\\
 $10^{-2}$ & $1$ & $1$\\
 $10^{-3}$ & $1$ & $1$\\
 $10^{-4}$ & $1$ & $2$\\
 $10^{-5}$ & $2$ & $3$\\
 $10^{-6}$ & $2$ & $5$\\
 $10^{-7}$ & $3$ & $9$\\
 $10^{-8}$ & $4$ & $16$\\
 $10^{-9}$ & $6$ & $28$\\
 $10^{-10}$ & $9$ & $50$\\
 $10^{-11}$ & $13$ & $89$\\
 $10^{-12}$ & $19$ & $158$\\
 $10^{-13}$ & $27$ & $280$\\
 $10^{-14}$ & $39$ & $498$\\
 $10^{-15}$ & $57$ & $884$\\
 $10^{-16}$ & $84$ & $1572$
\end{tabular}
\end{center}

\end{ex}

\begin{ex}
Let $f(x)=e^x$. We assume an accuracy $\varepsilon=10^{-8}$ and examine how many subdivisions are needed to approximate the integral $\int_0^b e^x dx$ for $b\in\{1,2,\dots,10\}$. The results are also compared with the algorithm described in \cite{Was20}.

\begin{center}
\begin{tabular}{lp{4cm}p{4cm}}
 \hline
 $f(x)=e^x$ & \multicolumn{2}{c}{Number of subdivisions of the interval}\\\cline{2-3}
 Integration interval & Algorithm based\newline on Corollary~\ref{cor:adaptive}  & Algorithm for\newline $3$-convex functions,\newline see~\cite[Experiment 2]{Was20}\\
 \hline
 $[0,1]$ & $2$ & $12$\\
 $[0,2]$ & $5$ & $33$\\
 $[0,3]$ & $9$ & $64$\\
 $[0,4]$ & $14$ & $111$\\
 $[0,5]$ & $21$ & $178$\\
 $[0,6]$ & $29$ & $275$\\
 $[0,7]$ & $40$ & $412$\\
 $[0,8]$ & $54$ & $604$\\
 $[0,9]$ & $71$ & $872$\\
 $[0,10]$ & $93$ & $1244$
\end{tabular}
\end{center}

\end{ex}

Both experiments demonstrate that the adaptive method proposed in this paper requires significantly fewer computations than the analogous method from \cite{Was20}.

\bibliographystyle{plain}
\bibliography{wasowicz}

\end{document}